\documentclass[a4paper,10pt]{article}

\usepackage{amsmath,amssymb,amsthm,amscd}
\usepackage[mathscr]{eucal}
\usepackage{multirow}
\usepackage{hhline}
\usepackage[all]{xy}

\makeatletter
 
  \@addtoreset{equation}{subsection}
\makeatother

\newcommand{\plim}{\varprojlim}
\newcommand{\mcal}{\mathcal}

\newcommand{\mfrak}{\mathfrak}
\newcommand{\mbb}{\mathbb}
\newcommand{\mrm}{\mathrm}
\newcommand{\mfS}{\mathfrak{S}}
\newcommand{\vphi}{\varphi}
\newcommand{\mfL}{\mathfrak{L}}
\newcommand{\mfM}{\mathfrak{M}}
\newcommand{\mfN}{\mathfrak{N}}

\newcommand{\mft}{\mathfrak{t}}
\newcommand{\mfm}{\mathfrak{m}}
\newcommand{\wh}{\widehat}
\newcommand{\whR}{\widehat{\mathcal{R}}}

\newcommand{\Max}{\mathrm{Max}}

\newtheorem{theorem}{Theorem}

\newtheorem{lemma}[theorem]{Lemma}
\newtheorem{question}[theorem]{Question}
\newtheorem{proposition}[theorem]{Proposition}

\theoremstyle{definition}
\newtheorem{definition}[theorem]{Definition}
\newtheorem{remark}[theorem]{Remark}

\newcommand{\e}{\varepsilon}

\newcommand{\cO}{\mathcal{O}}

\newtheorem*{acknowledgements}{Acknowledgements}

\title{Full faithfulness theorem for torsion crystalline representations}

\author{Yoshiyasu Ozeki\footnote{
Research Institute for Mathematical Sciences, Kyoto University,
Kyoto 606-8502, JAPAN.
\endgraf
e-mail: {\tt yozeki@kurims.kyoto-u.ac.jp}
\endgraf
Partly supported by the Grant-in-Aid for Research Activity Start-up, 
The Ministry of Education, Culture, Sports, Science and Technology, Japan.}}

\date{}
\begin{document}
\maketitle

\begin{abstract}
Mark Kisin proved that a certain ``restriction functor'' on crystalline $p$-adic representations is fully faithful. 
In this paper, we prove the torsion analogue of Kisin's theorem.
\end{abstract}



\section{Introduction}
Let $p>2$ be a prime number and $r,r'\ge 0 $ integers.
Let $K$ be a complete discrete valuation field 
of mixed characteristic $(0,p)$ with perfect residue field
and absolute ramification index $e$.
Let $\pi=\pi_0$ be a uniformizer of $K$ and 
$\pi_n$ a $p^n$-th root of $\pi$ such that $\pi^p_{n+1}=\pi_n$ for all $n\ge 0$.  
Put $K_{\infty}=\bigcup_{n\ge 0}K(\pi_n)$ and 
denote by $G_K$ and $G_{\infty}$ absolute Galois groups of $K$ and $K_{\infty}$,
respectively.
In Theorem (0.2) of \cite{Kis},
Kisin proved that 
the functor ``restriction to $G_{\infty}$'' from crystalline $\mbb{Q}_p$-representations of $G_K$ 
to $\mbb{Q}_p$-representations of $G_{\infty}$
is fully faithful, 
which was a conjecture of Breuil (\cite{Br1}).
Hence we may say that 
crystalline $\mbb{Q}_p$-representations of $G_K$ are characterized by 
their restriction to $G_{\infty}$.
It should be noted that there exists an established theory describing 
representations of $G_{\infty}$ by easy linear algebra data, which is called \'etale $\vphi$-modules, introduced by Fontaine (\cite{Fo} A 1.2). 
In this paper, we are interested in the torsion analogue of the above Kisin's result.
For example, 
Breuil proved in Theorem 3.4.3 of \cite{Br2} that the functor ``restriction to $G_{\infty}$'' 
from finite flat representations
of $G_K$  to torsion $\mbb{Z}_p$-representations of $G_{\infty}$ is fully faithful 
(Remark \ref{Rem} (2)).
Our main theorem is motivated by his result:

\begin{theorem}
\label{Main}
Suppose  $er<p-1$ and $e(r'-1)<p-1$.
Let $T$ $($resp.\ $T')$ be a torsion crystalline $\mbb{Z}_p$-representation of $G_K$
with Hodge-Tate weights in $[0,r]$ $($resp.\ $[0,r'])$.
Then any $G_{\infty}$-equivalent morphism $T\to T'$ is in fact $G_K$-equivalent.

In particular, 
the functor from torsion crystalline $\mbb{Z}_p$-representations of $G_K$
with Hodge-Tate weights in $[0,r]$ to torsion $\mbb{Z}_p$-representations of $G_{\infty}$,
obtained by restricting the action of $G_K$ to $G_{\infty}$, is fully faithful.
\end{theorem}

\noindent
Here a torsion $\mbb{Z}_p$-representation of $G_K$ is said to be 
{\it torsion crystalline with Hodge-Tate weights in $[0,r]$}
if it can be written as the quotient of two lattices in some 
crystalline $\mbb{Q}_p$-representation of $G_K$ with Hodge-Tate weights in $[0,r]$.
For example, a torsion $\mbb{Z}_p$-representation of $G_K$ is finite flat if and only if 
it is torsion crystalline with Hodge-Tate weights in $[0,1]$ (Remark \ref{Rem} (2)).
If $e=1$, the latter part of Theorem \ref{Main} has been proven 
by Breuil via Fontaine-Laffaille theory (Remark \ref{Rem} (3)). 
On the other hand,  our proof is based on results on Kisin modules and $(\vphi,\hat{G})$-modules
(the notion of $(\vphi,\hat{G})$-modules is introduced in \cite{Li2}).
More precisely, we use maximal models for Kisin modules 
introduced in \cite{CL1} and results on ``the range of monodromy'' for $(\vphi,\hat{G})$-modules
given in Section 4 of \cite{GLS}. 

It seems natural to have the question whether
the condition ``$er<p-1$'' in the latter part of Theorem \ref{Main} is necessary and sufficient for the full faithfulness or not.
In fact, we know that the condition ``$er<p-1$'' is not necessary since our restriction functor 
is fully faithful for any $e$ when $r=1$ (Remark \ref{Rem} (2)).
(Maybe the necessary and sufficient condition for the full faithfulness is 
``$e(r-1)<p-1$'' (Remark \ref{Rem}).)
In addition, in the last section, we give some examples such that the restriction functor appeared in Theorem  \ref{Main}
is not full under some choices of  $K$ and $r$ which do not satisfy ``$er<p-1$''
(more precisely, ``$e(r-1)<p-1$'').
Examples are mainly given by using two methods:
The first one is direct computations of Galois cohomologies, which is a purely local method. 
The second one is based on the classical Serre's modularity conjecture, which is a global method.

\begin{acknowledgements}
It is a pleasure to thank Wansu Kim  
for useful
comments and correspondences to Theorem \ref{Main}. 
The author thanks Naoki Imai and Akio Tamagawa 
who gave him useful advice in the proof of his main theorem.
The author thanks also Keisuke Arai, Shin Hattori, Yuichiro Taguchi and Seidai Yasuda for their helpful comments on Proposition \ref{Prop3}. 
This work is supported by the Grant-in-Aid for Young Scientists Start-up.
\end{acknowledgements}


\section{Preliminaries}

Throughout this paper, 
we fix a prime number $p>2$.
Let $r\ge 0$ be an integer.
Let $k$ be a perfect field of 
characteristic $p$,
$W(k)$ its ring of Witt vectors, 
$K_0=W(k)[1/p]$, $K$ a finite totally 
ramified extension of $K_0$,
$\overline{K}$ a fixed algebraic closure of $K$
and $G_K=\mrm{Gal}(\overline{K}/K)$.
Fix a uniformizer $\pi\in K$ 
and denote by $E(u)$ its 
Eisenstein polynomial over $K_0$.
For any integer $n\ge 0$,
let $\pi_n\in \overline{K}$ be 
a $p^n$-th root of $\pi$ such that 
$\pi^p_{n+1}=\pi_n$.
Let $K_{\infty}=\bigcup_{n\ge 0}K(\pi_n)$ and 
$G_{\infty}=\mrm{Gal}(\overline{K}/K_{\infty})$.

For any topological group $H$,
we denote by 
$\mrm{Rep}_{\mrm{tor}}(H)$ 
(resp.\ $\mrm{Rep}_{\mbb{Z}_p}(H)$)
the category of finite torsion $\mbb{Z}_p$-representations 
of $H$
(resp.\ the category of 
finite free $\mbb{Z}_p$-representations of $H$).
We denote by $\mrm{Rep}^r_{\mbb{Z}_p}(G_K)$
the category of lattices in crystalline $\mbb{Q}_p$-representations 
of $G_K$ with Hodge-Tate weights in $[0,r]$. 
We say that $T\in \mrm{Rep}_{\mrm{tor}}(G_K)$ 
is {\it torsion crystalline with Hodge-Tate weights in $[0,r]$}
if it can be written as the quotient of $L'\subset L$ in $\mrm{Rep}^r_{\mbb{Z}_p}(G_K)$,
and denote by $\mrm{Rep}^r_{\mrm{tor}}(G_K)$ the category of them.

Let $R=\plim \cO_{\overline{K}}/p$ 
where $\cO_{\overline{K}}$ is 
the integer ring of $\overline{K}$
and the transition maps are 
given by the $p$-th power map.
Write
\underbar{$\pi$} $=(\pi_n)_{n\ge 0}\in R$
and let $[$\underbar{$\pi$}$]\in W(R)$ 
be the Teichm\"uller
representative of \underbar{$\pi$}.
Let $\mfS=W(k)[\![u]\!]$ equipped 
with a Frobenius endomorphism
$\varphi$ given by $u\mapsto u^p$ and 
the Frobenius on $W(k)$. 
We embed the $W(k)$-algebra $W(k)[u]$ into $W(R)$
via the map $u\mapsto [$\underbar{$\pi$}$]$.
This embedding extends to an 
embedding $\mfS\hookrightarrow W(R)$, 
which is compatible with Frobenius endomorphisms.

A {\it $\vphi$-module} ({\it over $\mfS$}) 
is an $\mfS$-module 
$\mfM$ equipped with a $\vphi$-semilinear map 
$\vphi\colon \mfM\to \mfM$.
A morphism between two $\vphi$-modules 
$(\mfM_1,\vphi_1)$ and $(\mfM_2,\vphi_2)$
is an $\mfS$-linear map $\mfM_1\to \mfM_2$
compatible with $\vphi_1$ and $\vphi_2$.
Denote by $'\mrm{Mod}^r_{/\mfS}$
the category of $\vphi$-modules $(\mfM,\vphi)$ 
{\it of height $\le r$}  
in the sense that $\mfM$ is of finite type 
over $\mfS$ and the cokernel of 
$1\otimes \vphi\colon \mfS\otimes_{\vphi,\mfS}\mfM\to \mfM$
is killed by $E(u)^r$.
Let $\mrm{Mod}^r_{/\mfS_{\infty}}$
be the full subcategory of $'\mrm{Mod}^r_{/\mfS}$
consisting of finite $\mfS$-modules 
which are killed by some power of $p$ and have projective dimension $1$ 
in the sense that $\mfM$ has a two term resolution  by
finite free $\mfS$-modules.
Let $\mrm{Mod}^{r}_{/\mfS}$
be the full subcategory of 
$'\mrm{Mod}^r_{/\mfS}$
consisting of finite free $\mfS$-modules.
We call an object of $\mrm{Mod}^r_{/\mfS_{\infty}}$ (resp.\ $\mrm{Mod}^r_{/\mfS}$)
a {\it torsion Kisin module} (resp.\ a {\it free Kisin module}).
A {\it Kisin module} is a torsion Kisin module or a free Kisin module.
For any Kisin module $\mfM$,
we define a $\mbb{Z}_p$-representation $T_{\mfS}(\mfM)$ of $G_{\infty}$ by 
\begin{align*}
T_{\mfS}(\mfM)=
\left\{
\begin{array}{ll}
\mrm{Hom}_{\mfS,\vphi}(\mfM,\mbb{Q}_p/\mbb{Z}_p\otimes_{\mbb{Z}_p}\mfS^{\mrm{ur}})\quad {\rm if}\ \mfM\ {\rm is}\ {\rm torsion}
\cr
\mrm{Hom}_{\mfS,\vphi}(\mfM,\mfS^{\mrm{ur}})\quad {\rm if}\ \mfM\ {\rm is}\ {\rm free}.  
\end{array}
\right.
\end{align*}
\noindent
Here, a $G_{\infty}$-action on 
$T_{\mfS}(\mfM)$ is given by 
$(\sigma.f)(x)=\sigma(f(x))$ 
for $\sigma\in G_{\infty}, f\in T_{\mfS}(\mfM), x\in \mfM$.

Here we recall the theory of Liu's $(\vphi,\hat{G})$-modules (cf.\ \cite{Li2}). 
Let $S$ be the $p$-adic completion 
of the divided power envelope of $W(k)[u]$ with respect to the ideal 
generated by $E(u)$.  
There exists a unique Frobenius map $\vphi\colon S\to S$
defined by $\varphi(u)=u^p$.
Put $S_{K_0}=S[1/p]=K_0\otimes_{W(k)} S$.
The inclusion $W(k)[u]\hookrightarrow W(R)$ 
via the map $u\mapsto [$\underbar{$\pi$}$]$
induces $\vphi$-compatible inclusions 
$\mfS\hookrightarrow S\hookrightarrow A_{\mrm{cris}}$
and $S_{K_0}\hookrightarrow B^+_{\mrm{cris}}$.
Fix a choice of primitive $p^i$-root 
of unity $\zeta_{p^i}$ for $i\ge 0$
such that $\zeta^p_{p^{i+1}}=\zeta_{p^i}$.
Put \underbar{$\e$} $=(\zeta_{p^i})_{i\ge 0}\in R^{\times}$
and $t=\mrm{log}([$\underbar{$\e$}$])\in A_{\mrm{cris}}$.
Denote by $\nu\colon W(R)\to W(\overline{k})$ 
a unique lift of the projection $R\to \overline{k}$,
which extends to a map 
$\nu \colon B^+_{\mrm{cris}}\to W(\overline{k})[1/p]$.
For any subring $A\subset B^+_{\mrm{cris}}$,
we put 
$I_+A=\mrm{Ker}(\nu\ \mrm{on}\  B^+_{\mrm{cris}})\cap A$.
For any integer $n\ge 0$,
let $t^{\{n\}}=t^{r(n)}\gamma_{\tilde{q}(n)}(\frac{t^{p-1}}{p})$ 
where $n=(p-1)\tilde{q}(n)+r(n)$ with $\tilde{q}(n)\ge 0,\ 0\le r(n) <p-1$
and $\gamma_i(x)=\frac{x^i}{i!}$ is 
the standard divided power.
We define a subring $\mcal{R}_{K_0}$ of $B^+_{\mrm{cris}}$
as below:
\[
\mcal{R}_{K_0}=\{\sum^{\infty}_{i=0} f_it^{\{i\}}\mid f_i\in S_{K_0}\
\mrm{and}\ f_i\to 0\ \mrm{as}\ i\to \infty\}.
\]
Put $\wh{\mcal{R}}=\mcal{R}_{K_0}\cap W(R)$
and $I_+=I_+\wh{\mcal{R}}$.
Put $\hat{K}=\bigcup_{n\ge 0} K_{\infty}(\zeta_{p^n})$ and 
$\hat{G}=\mrm{Gal}(\hat{K}/K)$.
Lemma 2.2.1 in \cite{Li2} shows that $\wh{\mcal{R}}$ $($resp.\ $\mcal{R}_{K_0})$ 
is a $\vphi$-stable $\mfS$-algebra
as a subring in $W(R)$ $($resp.\ $B^+_{\mrm{cris}})$, and $\nu$ induces 
$\mcal{R}_{K_0}/I_+\mcal{R}_{K_0}\simeq K_0$ and 
$\wh{\mcal{R}}/I_+\simeq S/I_+S\simeq \mfS/I_+\mfS\simeq W(k)$.
Furthermore, $\wh{\mcal{R}}, I_+, \mcal{R}_{K_0}$ and $I_+\mcal{R}_{K_0}$ are $G_K$-stable, and 
$G_K$-actions on them
factors through $\hat{G}$.
For any Kisin module $\mfM$,
we equip $\whR\otimes_{\vphi,\mfS} \mfM$ with 
a Frobenius by
$\vphi_{\wh{\mcal{R}}}\otimes \vphi_{\mfM}$.
It is known that the natural map
$
\mfM\rightarrow \wh{\mcal{R}}\otimes_{\vphi, \mfS} \mfM
$
given by $x\mapsto 1\otimes x$ is an injection (\cite{CL2}, Section 3.1).
By this injection, 
we regard $\mfM$ as a $\vphi(\mfS)$-stable 
submodule of $\whR\otimes_{\vphi,\mfS} \mfM$.

\begin{definition}
\label{Liumod}
A {\it $(\vphi, \hat{G})$-module} $(${\it of height} $\le r$$)$ 
is a triple $\hat{\mfM}=(\mfM, \vphi_{\mfM}, \hat{G})$ where
\begin{enumerate}
\item[(1)] $(\mfM, \vphi_{\mfM})$ is a 
           Kisin module (of height $\le r$), 

\vspace{-2mm}
           
\item[(2)] $\hat{G}$ is an $\wh{\mcal{R}}$-semilinear
           $\hat{G}$-action on $\whR\otimes_{\vphi, \mfS} \mfM$,
           
\vspace{-2mm}           
           
\item[(3)] the $\hat{G}$-action commutes with $\vphi_{\wh{\mcal{R}}}\otimes \vphi_{\mfM}$,

\vspace{-2mm}

\item[(4)] $\mfM\subset (\whR\otimes_{\vphi,\mfS} \mfM)^{H_K}$ where $H_K=\mrm{Gal}(\hat{K}/K_{\infty})$,

\vspace{-2mm}

\item[(5)] $\hat{G}$ acts on the 
           $W(k)$-module $(\whR\otimes_{\vphi,\mfS} \mfM)/I_+(\whR\otimes_{\vphi,\mfS} \mfM)$ 
           trivially.           
\end{enumerate}
If $\mfM$ is a torsion 
(resp.\ free) Kisin module,
we call $\hat{\mfM}$ a 
{\it torsion} 
(resp.\ {\it free}) {\it $(\vphi, \hat{G})$-module}.
\end{definition}

A morphism 
between two $(\vphi, \hat{G})$-modules $\hat{\mfM}_1=(\mfM_1, \vphi_1, \hat{G})$ 
and $\hat{\mfM}_2=(\mfM_2, \vphi_2, \hat{G})$ is a morphism 
$f\colon \mfM_1\to \mfM_2$ of $\vphi$-modules 
such that 
$\wh{\mcal{R}}\otimes f\colon \whR\otimes_{\vphi,\mfS} \mfM_1\to \whR\otimes_{\vphi,\mfS} \mfM_2$
is $\hat{G}$-equivalent.
We denote by $\mrm{Mod}^{r,\hat{G}}_{/\mfS_{\infty}}$
(resp.\ $\mrm{Mod}^{r,\hat{G}}_{/\mfS}$) the 
category of torsion $(\vphi, \hat{G})$-modules of height $\le r$
(resp.\ free $(\vphi, \hat{G})$-modules of height $\le r$).
We often regard $\whR\otimes_{\vphi,\vphi} \mfM$ as a $G_K$-module 
via the projection $G_K\twoheadrightarrow \hat{G}$. 
A sequence
$0\to \hat{\mfM}' \to \hat{\mfM} \to \hat{\mfM}''\to 0$
of $(\vphi,\hat{G})$-modules is {\it exact} if it is exact as $\mfS$-modules.
For a $(\vphi, \hat{G})$-module $\hat{\mfM}$,
we define a $\mbb{Z}_p$-representation $\hat{T}(\hat{\mfM})$ of $G_K$ by
\begin{align*}
\hat{T}(\hat{\mfM})=
\left\{
\begin{array}{ll}
\mrm{Hom}_{\wh{\mcal{R}},\vphi}(\whR\otimes_{\vphi,\mfS} \mfM, \mbb{Q}_p/\mbb{Z}_p\otimes_{\mbb{Z}_p}W(R))\quad 
{\rm if}\ \mfM\ {\rm is}\ {\rm torsion}  
\cr
\mrm{Hom}_{\wh{\mcal{R}},\vphi}(\whR\otimes_{\vphi,\mfS} \mfM, W(R))\quad {\rm if}\ \mfM\ {\rm is}\ {\rm free}. 
\end{array}
\right.
\end{align*}
\noindent
Here,  $G_K$ 
acts on $\hat{T}(\hat{\mfM})$ by $(\sigma.f)(x)=\sigma(f(\sigma^{-1}(x)))$
for $\sigma\in G_K,\ f\in \hat{T}(\hat{\mfM}),\ x\in \whR\otimes_{\vphi,\mfS}\mfM$. 
Then, there exists a natural $G_{\infty}$-equivalent map 
\[
\theta\colon T_{\mfS}(\mfM)\to \hat{T}(\hat{\mfM})
\]
defined by 
$\theta(f)(a\otimes m)=a\vphi(f(m))$ for 
$f\in T_{\mfS}(\mfM),\ a\in \wh{\mcal{R}}, m\in \mfM$.

Fix a topological generator $\tau$ of $\mrm{Gal}(\hat{K}/K_{p^{\infty}})$
where $K_{p^{\infty}}=\bigcup_{n\ge 0} K(\zeta_{p^n})$.
We may suppose that $\zeta_{p^n}=\tau(\pi_n)/\pi_n$ for all $n$, and this implies 
$\tau(u)=[\underline{\e}]u$ in $W(R)$.
There exists $\mft\in W(R)\smallsetminus pW(R)$ such that 
$\vphi(\mft)=pE(0)^{-1}E(u)\mft$. 
Such $\mft$ is unique up to units of $\mbb{Z}_p$ (cf.\ Example 2.3.5 of \cite{Li1}).  
The following theorems play important rolls  in the proof of 
Theorem \ref{Main}. 

\begin{theorem}[\cite{Li2}]
\label{Th1}
$(1)$ The map
$\theta\colon T_{\mfS}(\mfM)\to \hat{T}(\hat{\mfM})$
is an isomorphism.

\noindent
$(2)$ The contravariant functor $\hat{T}$ induces an anti-equivalence 
between the category $\mrm{Mod}^{r,\hat{G}}_{\mfS}$ 
of free $(\vphi,\hat{G})$-modules  of height $\le r$
and 
the category
of $G_K$-stable $\mbb{Z}_p$-lattices in semi-stable 
$\mbb{Q}_p$-representations of $G_K$ with Hodge-Tate weights in $[0,r]$.
\end{theorem}

\begin{theorem}[\cite{CL2}, Theorem 3.1.3 (4),  \cite{GLS}, Proposition 5.9]
\label{Th3}
Let $T\in \mrm{Rep}^r_{\mrm{tor}}(G_K)$
and take $L'\subset L$ in $\mrm{Rep}^r_{\mbb{Z}_p}(G_K)$ such that $T\simeq L/L'$.

\noindent
$(1)$ There exists an exact sequence 
$
\mathscr{S}\colon 0\to \hat{\mfL}\to \hat{\mfL}'\to \hat{\mfM}\to 0
$
of $(\vphi,\hat{G})$-modules
such that:
\begin{enumerate}
\item $\hat{\mfL}$ and $\hat{\mfL}'$ are free $(\vphi,\hat{G})$-modules of height $\le r$,
\item $\hat{\mfM}$ is a torsion $(\vphi,\hat{G})$-module of height $\le r$,
\item $\hat{T}(\mathscr{S})$ is equivalent to the exact sequence 
$0\to L'\to L\to T\to 0$ of $\mbb{Z}_p[G_K]$-modules.
\end{enumerate}

\noindent
$(2)$ Let $\hat{\mfM}$ be as in $(1)$. 
For any $x\in \mfM$, we have
$
\tau(x)-x\in u^p\vphi(\mft)(W(R)\otimes_{\vphi,\mfS} \mfM).
$
\end{theorem}
\begin{proof}
The assertion (2) is an easy consequence of \cite{GLS}, Proposition 5.9.
Here is one remark: 
In {\it loc.\ cit}, $K$ is assumed to be a finite extension of $\mbb{Q}_p$, but 
arguments in Section 4.1 and 4.2 of  {\it loc.\ cit.} proceed 
even if $K$ is not only a finite extension of $\mbb{Q}_p$ but also 
any complete discrete valuation field 
of mixed characteristic $(0,p)$ with perfect residue field. 
\end{proof}


\section{Proof of Theorem \ref{Main}}

For any integer $\alpha\ge 0$, we denote by $\mfm^{\ge \alpha}_R$
the ideal of $R$ consisting of $a\in R$ with $v_R(a)\ge \alpha$,
where $v_R$ is a valuation of $R$ such that $v_R(\underline{\pi})=\frac{1}{e}$.
Note that, if we put $\tilde{\mft}=\mft$ mod $p\in R$, then 
$v_R(\tilde{\mft})=\frac{1}{p-1}$ since $\vphi(\tilde{\mft})\in \underline{\pi}^e\tilde{\mft}\cdot R^{\times}$
(recall the equation $\vphi(\mft)=pE(0)^{-1}E(u)\mft$).

We note that we have natural inclusions $\mfM\subset \mfS\otimes_{\vphi,\mfS} \mfM\subset \whR\otimes_{\vphi,\mfS} \mfM
\subset W(R)\otimes_{\vphi,\mfS} \mfM$ for any $\mfM\in \mrm{Mod}^r_{/\mfS_{\infty}}$.
Denote by
$\mrm{Mod}^{r,\hat{G},\mrm{cris}}_{/\mfS_{\infty}}$
the full subcategory of $\mrm{Mod}^{r,\hat{G}}_{/\mfS_{\infty}}$
consisting of 
torsion $(\vphi,\hat{G})$-modules $\hat{\mfM}$ 
which satisfy the following; 
for any $x\in \mfM$,
\[
\tau(x)-x\in u^p\vphi(\mft)(W(R)\otimes_{\vphi,\mfS} \mfM).
\]

\noindent
(The subscript ``cris'' of $\mrm{Mod}^{r,\hat{G},\mrm{cris}}_{/\mfS_{\infty}}$ 
is plausible; see Theorem \ref{app} in Appendix.)
We define the full subcategory $\mrm{Rep}^{r,\hat{G},\mrm{cris}}_{\mrm{tor}}(G_K)$ of $\mrm{Rep}_{\mrm{tor}}(G_K)$ 
to be the essential image of 
the functor $\mrm{Mod}^{r,\hat{G},\mrm{cris}}_{/\mfS_{\infty}}
\subset \mrm{Mod}^{r,\hat{G}}_{/\mfS_{\infty}}\overset{\hat{T}}{\rightarrow} 
\mrm{Rep}_{\mrm{tor}}(G_K)$,
where $\hat{T}$ is defined in the previous section.
By Theorem \ref{Th3}, 
we have 
\[
\mrm{Rep}^r_{\mrm{tor}}(G_K)\subset \mrm{Rep}^{r,\hat{G},\mrm{cris}}_{\mrm{tor}}(G_K)
\]
and thus it follows Theorem \ref{Main}
from the following result.

\begin{theorem}
\label{Main'}
Suppose  $er<p-1$ and $e(r'-1)<p-1$.
Let $T\in \mrm{Rep}^{r,\hat{G},\mrm{cris}}_{\mrm{tor}}(G_K)$ and  
$T'\in \mrm{Rep}^{r',\hat{G},\mrm{cris}}_{\mrm{tor}}(G_K)$.
Then any $G_{\infty}$-equivalent morphism $T\to T'$ is in fact $G_K$-equivalent.
\end{theorem}

\begin{lemma}
\label{lem1}
Let $a\in W(R)\smallsetminus pW(R)$.
For any Kisin module $\mfM$,
the map 
\[
W(R)\otimes_{\vphi,\mfS}\mfM\to W(R)\otimes_{\vphi,\mfS}\mfM,\quad x\mapsto ax
\]
is injective.
\end{lemma}
\begin{proof}
We may suppose that $\mfM$ is a torsion Kisin module.
By a d\'evissage argument (\cite{Li1}, Proposition 2.3.2 (4)),
we may assume $p\mfM=0$. 
In this situation, the statement is clear since 
$W(R)\otimes_{\vphi,\mfS}\mfM$ is a finite direct sum of $R$.
\end{proof}

The following is a key lemma for our proof of Theorem \ref{Main'}:

\begin{lemma}
\label{lem2}
Let $r$ and $r'$ be non-negative integers with $e(r-1)<p-1$ $($without any assumption on $r'$$)$.
Let $\hat{\mfM}$ and $\hat{\mfN}$ be objects of $\mrm{Mod}^{r,\hat{G},\mrm{cris}}_{/\mfS_{\infty}}$ 
and $\mrm{Mod}^{r',\hat{G},\mrm{cris}}_{/\mfS_{\infty}}$, respectively.
Then we have
$\mrm{Hom}(\hat{\mfM}, \hat{\mfN})=\mrm{Hom}(\mfM, \mfN)$.

In particular, if $e(r-1)<p-1$, then
the forgetful functor 
$\mrm{Mod}^{r,\hat{G},\mrm{cris}}_{/\mfS_{\infty}}\to \mrm{Mod}^r_{/\mfS_{\infty}}$
is fully faithful.
\end{lemma}

The condition $e(r-1)<p-1$ is essential. See Remark \ref{rem2} below.

\begin{proof}
Let $f\colon \mfM\to \mfN$ be a morphism of Kisin modules and 
put $\hat{f}=W(R)\otimes f\colon W(R)\otimes_{\vphi,\mfS}\mfM\to W(R)\otimes_{\vphi,\mfS}\mfN$.
It suffices to prove that, for any $x\in \mfM$, 
$\Delta(1\otimes x)=0$ where $\Delta=\tau\circ\hat{f}-\hat{f}\circ\tau$.
We proceed by induction on $n$ such that $p^n\mfN=0$.

Suppose $n=1$, that is, $p\mfN=0$.
We may identify $W(R)\otimes_{\vphi,\mfS} \mfN$ with $R\otimes_{\vphi,\mfS} \mfN$. 
Since 
$
\Delta(1\otimes x)
=(\tau-1)(1\otimes f(x))-\hat{f}((\tau-1)(1\otimes x)),
$
we obtain the following implication
\begin{center}
$(0)$:\quad  For any $x\in \mfM$, 
$\Delta(1\otimes x)\in 
\mfm^{\ge c(0)}_R(R\otimes_{\vphi,\mfS} \mfN)$
\end{center}
where $c(0)=\frac{p}{p-1}+\frac{p}{e}$.
Note that 
\[
\Delta(1\otimes E(u)^rx)
=\tau(\vphi(E(u)))^r\Delta(1\otimes x)\\
=(\underline{\e} u)^{per}\Delta(1\otimes x)
\in R\otimes_{\vphi,\mfS} \mfN.
\]
On the other hand, since $\mfM$ is of height $\le r$,
we can write $E(u)^rx=\sum_{i\ge 0}a_i\vphi(y_i)$ for some $a_i\in \mfS$ and $y_i\in \mfM$.
Then we obtain
\begin{align*}
\Delta(1\otimes E(u)^rx)
=\sum_{i\ge 0}\tau(\vphi(a_i))\vphi(\Delta(1\otimes y_i))
\end{align*}
and it is contained in $\mfm^{\ge pc(0)}_R(R\otimes_{\vphi,\mfS} \mfN)$
by the implication $(0)$.
Since $R\otimes_{\vphi,\mfS} \mfN$ is free as an $R$-module, we obtain the implication
\begin{center}
$(1)$:\quad  For any $x\in \mfM$, 
$\Delta(1\otimes x)\in 
\mfm^{\ge c(1)}_R(R\otimes_{\vphi,\mfS} \mfN)$
\end{center}
where $c(1)=pc(0)-pr=\frac{p^2}{p-1}+\frac{p^2}{e}-pr$.
By repeating the same argument, for any $s\ge 0$,
we see the following implication
\begin{center}
$(s)$:\quad  For any $x\in \mfM$, 
$\Delta(1\otimes x)\in 
\mfm^{\ge c(s)}_R(R\otimes_{\vphi,\mfS} \mfN)$
\end{center}
where $c(s)=pc(s-1)-pr=\frac{p^{s+1}}{p-1}+\frac{p^{s+1}}{e}-p^sr-\cdots -pr$.
Since $e(r-1)<p-1$,
we know that $\mfm^{\ge c(s)}_R$
goes to zero when $s\to \infty$
and then we obtain $\Delta(1\otimes x)=0$. 

Suppose $n>1$. 
Consider the exact sequence 
$(\ast)\colon 0\to \mrm{Ker}(p)\to \mfN\overset{p}{\rightarrow} p\mfN\to 0$
of $\vphi$-modules.
By Lemma 2.3.1 and Proposition 2.3.2 of \cite{Li1},
we know that $\mfN':=\mrm{Ker}(p)$ and $\mfN'':=p\mfN$
are in $\mrm{Mod}^{r'}_{/\mfS_{\infty}}$.
Equipping $\whR\otimes_{\vphi,\mfS}\mfN''$ with $\hat{G}$-action
via the natural identification $p(\whR\otimes_{\vphi,\mfS}\mfN)=\whR\otimes_{\vphi,\mfS}\mfN''$,
we see that $\mfN''$ has a structure as a $(\vphi,\hat{G})$-module.
We can also equip $\whR\otimes_{\vphi,\mfS}\mfN'$ with $\hat{G}$-action
via the exact sequence
$0\to \whR\otimes_{\vphi,\mfS}\mfN'\to \whR\otimes_{\vphi,\mfS}\mfN\to \whR\otimes_{\vphi,\mfS}\mfN''\to 0$
(for the exactness, see \cite{CL2}, Lemma 3.1.2).
Since the sequence $0\to \whR/I_+\otimes_{\vphi,\mfS}\mfN'\to \whR/I_+\otimes_{\vphi\,\mfS}\mfN\to 
\whR/I_+\otimes_{\vphi,\mfS}\mfN''\to 0$
is also exact (\cite{Oz}, Corollary 2.11),
we know that $\mfN'$ also has a structure as a $(\vphi,\hat{G})$-module.
Summary, we obtained an exact sequence  
$0\to \hat{\mfN}'\to \hat{\mfN}\overset{p}{\rightarrow} \hat{\mfN}''\to 0$
in $\mrm{Mod}^{r',\hat{G}}_{/\mfS_{\infty}}$
whose underlying sequence of $\vphi$-modules is $(\ast)$.
Remark that $p \mfN'=0$ and $p^{n-1}\mfN''=0$.  
It is clear that $\hat{\mfN}''\in \mrm{Mod}^{r',\hat{G},\mrm{cris}}_{/\mfS_{\infty}}$.
Since $0\to \whR\otimes_{\vphi,\mfS} \mfN'\to \whR\otimes_{\vphi,\mfS} \mfN
\to \whR\otimes_{\vphi,\mfS} \mfN''\to 0$ is exact
and $p^{n-1}\mfN''=0$,
we obtain
$\Delta(1\otimes x)\in 
\whR\otimes_{\vphi,\mfS} \mfN'\subset W(R)\otimes_{\vphi,\mfS} \mfN'$
for any $x\in \mfM$ by the induction hypothesis. 
Moreover, 
we have in fact 
$\Delta(1\otimes x)\in 
u^p\vphi(\mft)(W(R)\otimes_{\vphi,\mfS} \mfN')$
since Lemma \ref{lem1} implies 
$(W(R)\otimes_{\vphi,\mfS} \mfN')\cap u^p\vphi(\mft)(W(R)\otimes_{\vphi,\mfS} \mfN)
=u^p\vphi(\mft)(W(R)\otimes_{\vphi,\mfS} \mfN')$.
Identifying 
$W(R)\otimes_{\vphi,\mfS} \mfN'$ with $R\otimes_{\vphi,\mfS} \mfN'$,
we obtain 
$\Delta(1\otimes x)\in 
\mfm^{\ge c(0)}_R(R\otimes_{\vphi,\mfS} \mfN').
$
By an analogous argument of the case where $n=1$,
we obtain the implication
\begin{center}
$(s)'$:\quad  For any $x\in \mfM$, 
$\Delta(1\otimes x)\in 
\mfm^{\ge c(s)}_R(R\otimes_{\vphi,\mfS} \mfN')$
\end{center}
for any $s\ge 0$ and this implies $\Delta(1\otimes x)=0$. 
\end{proof}

Before giving the proof of Theorem \ref{Main'},
we have to recall the theory of {\it maximal} Kisin modules.
Now we give a very rough sketch of it 
(for more precise information, see \cite{CL1}.
Our sketch here is the case where ``$r=\infty$'' in {\it loc}.\ {\it cit}.). 
For any $\mfM\in \mrm{Mod}^r_{/\mfS_{\infty}}$,
put $\mfM[1/u]=\mfS[1/u]\otimes_{\mfS} \mfM$ and 
denote by $F_{\mfS}(\mfM[1/u])$
the (partially) ordered set (by inclusion)
of torsion Kisin modules $\mfN$ of finite height which is contained in $\mfM[1/u]$ 
and $\mfN[1/u]=\mfM[1/u]$ as $\vphi$-modules.
Here, a torsion Kisin module is called {\it of finite height} if it is of  height $\le s$ 
for some integer $s\ge 0$. 
The set $F_{\mfS}(\mfM[1/u])$ has a greatest element (cf.\ {\it loc}.\ {\it cit}., Corollary 3.2.6),
which is denoted by $\Max(\mfM)$.
We say that $\mfM$ is {\it maximal}
if it is the greatest element of $F_{\mfS}(\mfM[1/u])$. 
The implication $\mfM\mapsto \Max(\mfM)$ defines a functor ``$\Max$''
from the category of torsion Kisin modules of finite height
into the category $\Max_{/\mfS_{\infty}}$ of maximal torsion Kisin modules.
Furthermore, the functor $T_{\mfS}\colon \Max_{/\mfS_{\infty}}\to \mrm{Rep}_{\mrm{tor}}(G_{\infty})$,
defined by 
$T_{\mfS}(\mfM)=\mrm{Hom}_{\mfS,\vphi}(\mfM,\mbb{Q}_p/\mbb{Z}_p\otimes_{\mbb{Z}_p} W(R))$,
is fully faithful (cf.\ {\it loc}.\ {\it cit}., Corollary 3.3.10). 
It is not difficult to check that
$T_{\mfS}(\Max(\mfM))$ is canonically isomorphic to $T_{\mfS}(\mfM)$ as representations of $G_{\infty}$
for any torsion Kisin module $\mfM$. 

\begin{lemma}
\label{Lem2}
Suppose $er<p-1$.
Then any $\mfM\in \mrm{Mod}^r_{/\mfS_{\infty}}$ is maximal.
\end{lemma}
\begin{proof}
We prove by induction on $n$ such that $p^n\mfM=0$.
If $n=1$, then the assertion follows by Lemma 3.3.4 of \cite{CL1}.
Suppose $n>1$ and $p^n\mfM=0$.
Take any $\mfN\in F_{\mfS}(\mfM[1/u])$ such that $\mfM\subset \mfN$ and put $M=\mfM[1/u]=\mfN[1/u]$.
Denote by $\mrm{pr}$ the natural surjection $M\to M/pM$.
Putting $\mfM'=pM\cap \mfM, \mfM''=\mrm{pr}(\mfM), \mfN'=pM\cap \mfN$ and 
$\mfN''=\mrm{pr}(\mfN)$,
we see that $\mfM'$ and $\mfM''$ are objects of $\mrm{Mod}^r_{/\mfS_{\infty}}$,
and $\mfN'$ and $\mfN''$ are torsion Kisin modules of finite height.
Furthermore, we see that natural sequences 
$0\to \mfM'\to \mfM\overset{\mrm{pr}}{\rightarrow} \mfM''\to 0$ 
and  $0\to \mfN'\to \mfN\overset{\mrm{pr}}{\rightarrow} \mfN''\to 0$
of $\vphi$-modules
are exact.
By the induction hypothesis,
we know that $\mfM'$ and $\mfM''$ are maximal and thus
$\mfN'=\mfM'$ and $\mfN''=\mfM''$ (remark that $\mfM'[1/u]=\mfN'[1/u]=pM$ 
and $\mfM''[1/u]=\mfN''[1/u]=M/pM$).
This implies $\mfN=\mfM$.
\end{proof}

\begin{proof}[Proof of Theorem \ref{Main'}]
Suppose that $er<p-1$ and $e(r'-1)<p-1$.
Let $T\in \mrm{Rep}^{r,\hat{G},\mrm{cris}}_{\mrm{tor}}(G_K)$ (resp.\ $T'\in \mrm{Rep}^{r',\hat{G},\mrm{cris}}_{\mrm{tor}}(G_K)$)
and take $\hat{\mfM}\in \mrm{Mod}^{r,\hat{G},\mrm{cris}}_{/\mfS_{\infty}}$ 
(resp.\ $\hat{\mfM}'\in \mrm{Mod}^{r',\hat{G},\mrm{cris}}_{/\mfS_{\infty}}$) such that  
$T=\hat{T}(\hat{\mfM})$ (resp.\ $T'=\hat{T}(\hat{\mfM}')$).
Note that $\mfM=\Max(\mfM)$ by Lemma \ref{Lem2}.
By Theorem \ref{Th1} (1),
we have the following commutative diagram:
\begin{center}
$\displaystyle \xymatrix{
\mrm{Hom}_{G_K}(T,T')\ar@{^{(}->}[rr] & &   
\mrm{Hom}_{G_{\infty}}(T,T') \\
\mrm{Hom}(\hat{\mfM}',\hat{\mfM}) \ar^{\hat{T}}[u] \ar^{\mrm{forgetful}}[r] &
\mrm{Hom}(\mfM',\mfM) \ar^{\Max\quad \  }[r] & \mrm{Hom}(\Max(\mfM'),\mfM). \ar^{T_{\mfS}}[u]. 
}$
\end{center}
The first bottom horizontal arrow is bijective by Lemma \ref{lem2} and 
the second is also by an easy argument.
Since the right vertical arrow is bijective,
the top horizontal arrow must be bijective. 
\end{proof}
\begin{remark}
\label{rem1}
By Lemma \ref{lem2}, we can prove the latter part of Theorem \ref{Main} directly 
without using the former part of Theorem \ref{Main} as below:
Suppose that $er<p-1$.
Let $T\in \mrm{Rep}^r_{\mrm{tor}}(G_K)$ (resp.\ $T'\in \mrm{Rep}^r_{\mrm{tor}}(G_K)$)
and take $\hat{\mfM}$ (resp.\ $\hat{\mfM}'$) be as in Theorem \ref{Th3}, 
which is an object of $\mrm{Mod}^{r,\hat{G},\mrm{cris}}_{/\mfS_{\infty}}$.
By Theorem \ref{Th1} (1),
we have the commutative diagram
\begin{center}
$\displaystyle \xymatrix{
\mrm{Hom}_{G_K}(T,T')\ar@{^{(}->}[r] &  
\mrm{Hom}_{G_{\infty}}(T,T') \\
\mrm{Hom}(\hat{\mfM}',\hat{\mfM}) \ar^{\hat{T}}[u] \ar^{\mrm{forgetful}}[r] &
\mrm{Hom}(\mfM',\mfM) \ar^{T_{\mfS}}[u] 
}$
\end{center}
and then we obtain the desired result by an analogous argument before this remark.
\end{remark}

\begin{remark}
\label{rem2}
The condition $e(r-1)<p-1$ in Lemma \ref{lem2} is essential for the fullness 
of the forgetful functor 
$\mrm{Mod}^{r,\hat{G},\mrm{cris}}_{/\mfS_{\infty}}\to \mrm{Mod}^r_{/\mfS_{\infty}}$
(note that this functor is always faithful).
In fact, we have an example which implies that
this forgetful functor is not full even if $e(r-1)=p-1$.
In the below, 
we show that the forgetful functor 
$\mrm{Mod}^{r,\hat{G},\mrm{cris}}_{/\mfS_{\infty}}\to \mrm{Mod}^r_{/\mfS_{\infty}}$ 
is not full when $K=\mbb{Q}_p$ and $r=p$.

Suppose $K=\mbb{Q}_p$. Let $E_{\pi}$ be the Tate curve over $\mbb{Q}_p$ associated with $\pi$.
Lemma \ref{Prop3lem} in the next section says that
the $2$-dimensional $\mbb{F}_p$-representation $E_{\pi}[p]$ of $G_{\mbb{Q}_p}$
is torsion crystalline with Hodge-Tate weights in $[0,p]$.
In particular, by Theorem \ref{Th3},
there exists a $(\vphi,\hat{G})$-module  $\hat{\mfM}\in \mrm{Mod}^{p,\hat{G},\mrm{cris}}_{/\mfS_{\infty}}$
such that  $\hat{T}(\hat{\mfM})\simeq E_{\pi}[p]$.
On the other hand, for any non-negative integer $\ell$,
define the $(\vphi,\hat{G})$-module $\hat{\mfS}_1(\ell)=(\mfS_1(\ell),\vphi, \hat{G})$ as below:
$\mfS_1(\ell)=k[\![u]\!]\cdot \mfrak{f}^{\ell}$ is the rank-1 free 
$k[\![u]\!]$-module equipped with the Frobenius 
$\vphi(\mfrak{f}^{\ell})=c_0^{-\ell}u^{e\ell}\cdot \mfrak{f}^{\ell}$,
and define a $\hat{G}$-action on $\whR\otimes_{\vphi, \mfS} \mfS_1(\ell)$
by $\tau(\mfrak{f}^{\ell})=\hat{c}^{\ell}\cdot \mfrak{f}^{\ell}$.
Here, 
$\hat{c}=\prod^{\infty}_{n=1}\vphi^n(\frac{E(u)}{\tau(E(u))})$, which  
is contained in $ \whR^{\times}$ (cf.\ Example 3.2.3 of \cite{Li3}).
Then Example 3.2.3 of {\it loc}. {\it cit}. says that 
$\hat{T}(\hat{\mfS}_1(\ell))\simeq \mbb{F}_p(\ell)$.
On the other hand,
we define the $(\vphi,\hat{G})$-module 
$\hat{\mfS}_1(\ell)_0=(\mfS_1(\ell)_0,\vphi, \hat{G})$ as below:
Put $\ell_0=\mrm{max}\{\ell'\in \mbb{Z}_{\ge 0};e\ell-(p-1)\ell'\ge 0 \}$.
We denote by $\mfS_1(\ell)_0=k[\![u]\!]\cdot \mfrak{g}^{\ell}$ the rank-1 free 
$k[\![u]\!]$-module equipped with the Frobenius 
$\vphi(\mfrak{f}^{\ell})=c_0^{-\ell}u^{e\ell-(p-1)\ell_0}\cdot \mfrak{g}^{\ell}$,
and define a $\hat{G}$-action on $\whR\otimes_{\vphi, \mfS} \mfS_1(\ell)$
by $\tau(\mfrak{g}^{\ell})=\underline{\e}^{-p\ell_0}\hat{c}^{\ell}\cdot \mfrak{g}^{\ell}$.
(The generator $\mfrak{g}^{\ell}$ is taken to behave as $u^{-\ell_0}\mfrak{f}^{\ell}$.)
Then we see that 
$\Max(\mfS_1(\ell))=\mfS_1(\ell)_0$, and 
$\hat{\mfS}_1(\ell)_0$ (and $\hat{\mfS}_1(\ell)$)
are objects of $\mrm{Mod}^{\ell,\hat{G},\mrm{cris}}_{/\mfS_{\infty}}$.
We also see $\hat{T}(\hat{\mfS}_1(\ell)_0)\simeq \hat{T}(\hat{\mfS}_1(\ell))\simeq \mbb{F}_p(\ell)$.
Now we consider the following commutative diagram
(here we remark that $\mfS_1(0)\oplus \mfS_1(1)_0$ is maximal):
\begin{center}
$\displaystyle \xymatrix{
\mrm{Hom}_{G_{\mbb{Q}_p}}(\mbb{F}_p\oplus \mbb{F}_p(1),E_{\pi}[p])\ar@{^{(}->}[rr] & &   
\mrm{Hom}_{G_{\infty}}(\mbb{F}_p\oplus \mbb{F}_p(1),E_{\pi}[p]) \\
\mrm{Hom}(\hat{\mfM}, \hat{\mfS}_1(0)\oplus \hat{\mfS}_1(1)_0) 
\ar^{\hat{T}}[u] \ar^{\mrm{forgetful}}[r] &
\mrm{Hom}(\mfM, \mfS_1(0)\oplus \mfS_1(1)_0) \ar^{\Max\quad \  }[r] & 
\mrm{Hom}(\Max(\mfM), \mfS_1(0)\oplus \mfS_1(1)_0). \ar^{T_{\mfS}}[u]. 
}$
\end{center}
\noindent
The second bottom horizontal arrow and the right vertical arrow are 
bijective since $\mfS_1(0)\oplus \mfS_1(1)_0$ is maximal.
On the other hand, it is well-known that the inclusion
$\mrm{Hom}_{G_{\mbb{Q}_p}}(\mbb{F}_p\oplus \mbb{F}_p(1),E_{\pi}[p])
\subset \mrm{Hom}_{G_{\infty}}(\mbb{F}_p\oplus \mbb{F}_p(1),E_{\pi}[p])$
is not equal.
Therefore, the first bottom horizontal arrow is not surjective.
This implies that the forgetful functor 
$\mrm{Mod}^{p,\hat{G},\mrm{cris}}_{/\mfS_{\infty}}\to \mrm{Mod}^p_{/\mfS_{\infty}}$ 
is not full.
\end{remark}

\begin{remark}
Combining Theorem (2.3.5) of \cite{Kis}, Theorem \ref{Th3} and Lemma \ref{lem2},
we see that 
the forgetful functor 
$\mrm{Mod}^{1,\hat{G},\mrm{cris}}_{/\mfS_{\infty}}\to \mrm{Mod}^1_{/\mfS_{\infty}}$
is an equivalence of categories.
\end{remark}

\section{Non-fullness: Examples}

In the previous section,
we showed that the restriction functor 
$\mrm{Rep}^r_{\mrm{tor}}(G_K)\overset{\mrm{res}}{\longrightarrow} \mrm{Rep}_{\mrm{tor}}(G_{\infty})$
is fully faithful under the condition that $er<p-1$.
However, the full faithfulness may not hold 
if $er\ge p-1$. In this section, we give some examples of this phenomenon. 
It should be noted that all our examples appearing in this section 
are given under the condition  $e(r-1)\ge p-1$.

Let $\mu_{p^{n}}$ be the set of $p^n$-th roots of unity in $\overline{K}$,
$\mu_{p^{\infty}}:=\bigcup_{n\ge 0}\mu_{p^{n}}$
and denote by $G_1\subset G_K$ the absolute Galois group of $K(\pi_1)$.
Remark that,
if the restriction functor $\mcal{C}\to \mrm{Rep}_{\mrm{tor}}(G_1)$ is not fully faithful
for a full subcategory $\mcal{C}$ of $\mrm{Rep}_{\mrm{tor}}(G_K)$,
then the restriction functor $\mcal{C}\to \mrm{Rep}_{\mrm{tor}}(G_{\infty})$ is not fully faithful.  
Furthermore, we also remark that restriction functors 
$\mcal{C}\to \mrm{Rep}_{\mrm{tor}}(G_{\infty})$ and 
$\mcal{C}\to \mrm{Rep}_{\mrm{tor}}(G_1)$
are always faithful.

\begin{proposition}
\label{Prop}
Let $K$ be a finite extension of $\mbb{Q}_p$.
Let $s$ be the largest integer $n$ such that 
$\mu_{p^n}\subset K$.
Suppose that $s\ge 1$ and 
$K(\mu_{p^{s+1}})/K$ is ramified. 
Then the functor from torsion crystalline $\mbb{Z}_p$-representations of $G_K$
with Hodge-Tate weights in $[0,p+1]$ to torsion $\mbb{Z}_p$-representations of $G_1$,
obtained by restricting the action of $G_K$ to $G_1$, is not full.
\end{proposition}

The lemma below follows from direct calculations.
\begin{lemma}
\label{lem3}
Let $s\ge 1$ be an integer and $\psi\colon G_K\to \mbb{Z}^{\times}_p$ an unramified character
with the property that $s$ is the largest integer $n$ such that $\psi$ mod $p^n$ is trivial.
Define $\beta_{\psi}\colon G_K\to \mbb{Z}_p$ by the relation 
$\psi=1+p^s\beta_{\psi}$
and put $\bar{\beta}_{\psi}=\beta$ mod $p$.
Denote by $\delta^0_{\psi}\colon H^0(G_K,\mbb{Q}_p/\mbb{Z}_p(\psi))\to H^1(G_K,\mbb{F}_p)$
the connection map coming from the exact sequence $0\to \mbb{F}_p\to \mbb{Q}_p/\mbb{Z}_p(\psi)
\overset{p}{\rightarrow}\mbb{Q}_p/\mbb{Z}_p(\psi)\to 0$
of $G_K$-modules.
Then $\bar{\beta}_{\psi}\in H^1(G_K,\mbb{F}_p)$ and 
$\mrm{Im}(\delta^0_{\psi})=\mbb{F}_p.\bar{\beta}_{\psi}$.
\end{lemma}
\begin{proof}[Proof of Proposition \ref{Prop}]
Let $\e\colon G_K\to \mbb{Z}^{\times}_p$ be the $p$-adic cyclotomic character and $\bar{\e}:=\e$ mod $p$
the mod $p$ cyclotomic character. 
Let $K$ and $s\ge 1$ be as in Proposition \ref{Prop}.  
Let $\chi\colon G_K\to \mbb{Z}^{\times}_p$ be an unramified character 
such that $\chi$ mod $p^s$ is trivial. 
It suffices to show that, for some choice of $\chi$, there exist $\rho\colon G_K\to GL_2(\mbb{Z}_p)$
and $2\le r\le p+1$ with an exact sequence  
$0\to \chi\e^r\to \rho\to 1\to 0$ of representations of $G_K$ such that 
$\rho$ mod $p$ is not trivial on $G_K$ but is trivial on $G_1$.
Here, $1$ in the above exact sequence means the trivial character. 
Since $\mu_p\subset K$, 
we can define $f_0\in H^1(G_K,\mbb{F}_p)$ such that 
$f_0$ factors through $\hat{G},\ f_0(\tau)=1$ and $f_0|_{H_K}=0$,
where $H_K$ is defined in Definition \ref{Liumod}.
The kernel of the restriction map $H^1(G_K,\mbb{F}_p)\to H^1(G_1,\mbb{F}_p)$
is a one dimensional $\mbb{F}_p$-vector space which is generated by $f_0$.
Let $H\subset H^1(G_K,\mbb{F}_p)$ be an annihilator of $f_0$ under the Tate paring.
For any integer $\ell$, 
denote by $\delta^1_{\chi,\ell}\colon H^1(G_K,\mbb{F}_p)\to H^2(G_K,\mbb{Z}_p(\chi\e^{\ell}))$
(resp.\ $\delta^0_{\chi,\ell}\colon H^1(G_K,\mbb{Q}_p/\mbb{Z}_p(\chi^{-1}\e^{1-\ell}))
\to H^1(G_K,\mbb{F}_p)$)
the connection map coming from the exact sequence 
$0\to \mbb{Z}_p(\chi\e^{\ell})\overset{p}{\rightarrow} \mbb{Z}_p(\chi\e^{\ell})\to \mbb{F}_p\to 0$
(resp.\ $0\to \mbb{F}_p\to\mbb{Q}_p/\mbb{Z}_p(\chi^{-1}\e^{1-\ell})
\overset{p}{\rightarrow} \mbb{Q}_p/\mbb{Z}_p(\chi^{-1}\e^{1-\ell})\to 0$)
of $G_K$-modules.
By Tate local duality,
the condition that $f_0$ lifts to $H^1(G_K,\mbb{Z}_p(\chi\e^{\ell}))$ 
is equivalent to the condition that $\mrm{Im}(\delta^0_{\chi,\ell})\subset H$.
Hence it is enough to choose $\chi$ which satisfies the latter condition
for some $2\le \ell\le p+1$.

Since $K(\mu_{p^{s+1}})/K$ is ramified, we know that
$s$ is the largest integer $n$ such that $\chi^{-1}\e^{-1}$ mod $p^n$ is trivial.
Take $\beta_{\chi^{-1}\e^{-1}}$ and $\bar{\beta}_{\chi^{-1}\e^{-1}}$
as in Lemma \ref{lem3}. For simplicity, we write 
$\alpha_{\chi}:=\beta_{\chi^{-1}\e^{-1}}$ and 
$\bar{\alpha}_{\chi}:=\bar{\beta}_{\chi^{-1}\e^{-1}}$.
By Lemma \ref{lem3}, $\mrm{Im}(\delta^0_{\chi,2})$ is generated by $\bar{\alpha}_{\chi}$.
If $\bar{\alpha}_1$ is contained in $H$, then we finish the proof 
(choose $\chi$ as the trivial character $1$). 
Suppose $\bar{\alpha}_1$ is not contained in $H$.
From now on, we fix $\chi$ as follows; $\chi$ is the unramified character $G_K\to \mbb{Z}^{\times}_p$
with $\chi(\mrm{Frob}_K)=(1+p^s)^{-1}$, where $\mrm{Frob}_K$ is the arithmetic Frobenius of $K$. 
Let $u_1\colon G_K\to \mbb{F}_p$ be the unramified homomorphism with $u_1(\mrm{Frob}_K)=1$.
Then we obtain $\bar{\alpha}_{\chi}=u_1+\bar{\alpha}_1$. 
Since $K(\mu_{p^{s+1}})/K$ is ramified, we see that 
$\bar{\alpha}_1|_{I_K}$ is not zero where $I_K$ is the inertia subgroup of $G_K$.
This implies $u_1\notin \mbb{F}_p.\bar{\alpha}_1$.
Noting that $H^1(G_K,\mbb{F}_p)=H\oplus \mbb{F}_p.\bar{\alpha}_1$,
we have $\bar{\alpha}_{\chi}+\bar{a}\bar{\alpha}_1\in H$ for some $\bar{a}\in \mbb{F}_p$.
Let $0\le a\le p-1$ be the integer such that $a$ mod $p$ is $\bar{a}$.
Under the modulo $p^{2s}$, we have 
$\chi^{-1}\e^{-(1+a)}=\chi^{-1}\e^{-1}\cdot \e^{-a}=
(1+p^s\alpha_{\chi})(1+p^sa\alpha_1)=1+p^s(\alpha_{\chi}+a\alpha_1)$.
Since $\bar{\alpha}_{\chi}+\bar{a}\bar{\alpha}_1=u_1+(\bar{a}+1)\bar{\alpha}_1\not =0$,
we see  that
$s$ is the largest integer $n$ such that $\chi^{-1}\e^{-(1+a)}$ mod $p^n$ is trivial.
Hence, defining  $\beta_{\chi^{-1}\e^{-(1+a)}}$ as in Lemma \ref{lem3},
we obtain $\bar{\beta}_{\chi^{-1}\e^{-(1+a)}}=\bar{\alpha}_{\chi}+\bar{a}\bar{\alpha}_1$.
Therefore,
we obtain that $\mrm{Im}(\delta^0_{\chi,2+a})=\mbb{F}_p.\bar{\beta}_{\chi^{-1}\e^{-(1+a)}}\subset H$ and we have done.
\end{proof}

Unfortunately,
Proposition \ref{Prop} can not be applied even if $K=\mbb{Q}_p$.
On the other hand, the following proposition is effective for $K=\mbb{Q}_p$,
but we need a certain restriction on the choice of the uniformizer $\pi$.  
Let $L$ be the unique degree $p$ extension of $K$ which is contained in $K(\mu_{p^{\infty}})$.

\begin{proposition}
\label{Prop2}
Let $K$ be a finite extension of $\mbb{Q}_p$. 
Suppose that $\pi$ is contained in $\mrm{Norm}_{L/K}(L^{\times})$.
$($Thus the extension $L/K$ must be totally ramified in this case.$)$
Then the functor from torsion crystalline $\mbb{Z}_p$-representations of $G_K$
with Hodge-Tate weights in $[0,p]$ to torsion $\mbb{Z}_p$-representations of $G_1$,
obtained by restricting the action of $G_K$ to $G_1$, is not full.
\end{proposition}

\begin{proof}
Let $s$ be the largest integer $n$ such that $\mu_{p^n}\subset K$.
Then we can write $\e^{1-p}=1+p\psi$ with some map $\psi\colon G_K\to \mbb{Z}_p$. 
Putting $\bar{\psi}=\psi$ mod $p\colon G_K\to \mbb{F}_p$, 
we see that $\bar{\psi}$ is non-trivial homomorphism with kernel $\mrm{Gal}(\overline{K}/L)$.
Let $\delta^0\colon H^1(G_K, \mbb{Q}_p/\mbb{Z}_p(1-p))\to H^2(G_K,\mbb{F}_p)$)
be the connection map arising from the exact sequence 
$0\to \mbb{F}_p\to  \mbb{Q}_p/\mbb{Z}_p(1-p) \overset{p}{\rightarrow} \mbb{Q}_p/\mbb{Z}_p(1-p)\to 0$.
Under
the isomorphism $K^{\times}/(K^{\times})^p\simeq H^1(G_K, \mbb{F}_p(1))$ via 
Kummer theory,
$\pi$ mod $(K^{\times})^p$
corresponds to the $1$-cocycle $[\pi]$ defined by 
$\sigma\mapsto \frac{\sigma(\pi_1)}{\pi_1}$ 
for $\sigma\in G_K$, which is clearly trivial on $G_1$. 
By Tate local duality and the fact that the image of $\delta^0$ is generated by $\bar{\psi}$ (cf.\ Lemma \ref{lem3}),
it suffices to show that $([\pi], \bar{\psi})$ maps to zero under the Tate pairing 
$H^1(G_K,\mbb{F}_p(1))\times H^1(G_K,\mbb{F}_p)\to \mbb{Q}/\mbb{Z}$
(in fact, this implies that $[\pi]$ lifts to $H^1(G_K,\mbb{Z}_p(p))$
and we obtain the desired result).
Let 
$\phi_{L/K}\colon K^{\times}/\mrm{Norm}_{L/K}(L^{\times})\overset{\sim}{\rightarrow}\mrm{Gal}(L/K)$
be the isomorphism of local class field theory.
It is enough to show that 
$\bar{\psi}(\phi_{L/K}(\pi))=0$.
Our assumption of $\pi$ implies that this equality certainly holds.
\end{proof}

Now we give an example for the non-fullness of our restriction functor
without any assumption on the choice of the uniformizer $\pi$.

\begin{proposition}
\label{Prop3}
The functor from torsion crystalline $\mbb{Z}_p$-representations of $G_{\mbb{Q}_p}$
with Hodge-Tate weights in $[0,p]$ to torsion $\mbb{Z}_p$-representations of $G_1$,
obtained by restricting the action of $G_{\mbb{Q}_p}$ to $G_1$, is not full.
\end{proposition}

Let $v_K$ be a valuation of $K$ normalized such that $v_K(K^{\times})=\mbb{Z}$,
and we write $v_p:=v_{\mbb{Q}_p}$.

\begin{lemma}
\label{Prop3lem}
Let $K$ be a finite extension of $\mbb{Q}_p$ and $q\in \mbb{Q}^{\times}_p(K^{\times})^p$.
Let $E_q$ be the Tate curve over $K$ associated with $q$.
If $p\nmid v_K(q)$, then $E_q[p]$ is 
torsion crystalline with Hodge-Tate weights in $[0,p]$.
\end{lemma}

\begin{proof}
We have a decomposition $q=q'q''$, where $q'\in \mbb{Q}^{\times}_p, v_K(q')>0$ 
and $q''\in (K^{\times})^p$.
Let $E_{q'}$ be the Tate curve over $\mbb{Q}_p$ associated with $q'$.
Then $E_{q'}[p]$ is a representation of $G_{\mbb{Q}_p}$ and 
we have an isomorphism $E_q[p]\simeq (E_{q'}[p])|_{G_K}$.
Hence we can reduce the case where $K=\mbb{Q}_p$.
Let $\ell>3$ be a prime number different from $p$ such that $-\ell$ 
is not a square in $\mbb{F}^{\times}_p$ (recall that $p$ is odd). 
Choose an elliptic curve $E_{(\ell)}$ over $\mbb{Q}_{\ell}$
which has good supersingular reduction.
Since $\ell>3$,  we have $\# E_{(\ell)}(\mbb{F}_{\ell})=1+\ell$.  
Thus the characteristic polynomial of $E_{(\ell)}[p]$ for 
the arithmetic Frobenius of $\ell$ is $X^2+\ell\in \mbb{F}_p[X]$,
which does not have a root in $\mbb{F}_p$.
Hence $E_{(\ell)}[p]$  
is an irreducible representation of $G_{\mbb{Q}_{\ell}}$
where $G_{\mbb{Q}_{\ell}}$ is the absolute Galois group of $\mbb{Q}_{\ell}$.
We define $\mcal{S}$ to be the set of $\mbb{Q}$-isomorphism classes of elliptic curves $E$ defined over $\mbb{Q}$
which satisfy the following conditions:
\vspace{-2mm}
\begin{enumerate}
\item[(a)] $E$ has multiplicative reduction at $p$ and $v_p(j(E))=v_p(j(E_q))(=-v_p(q))$ 
where $j(E)$ is the $j$-invariant of $E$;
\vspace{-2mm}

\item[(b)] $E[p]\simeq E_q[p]$ 
as $\mbb{F}_p$-representations of $G_{\mbb{Q}_p}$;
\vspace{-2mm}

\item[(c)] $E[p]\simeq E_{(\ell)}[p]$ 
as $\mbb{F}_p$-representations of $G_{\mbb{Q}_{\ell}}$.
\end{enumerate}
\vspace{-2mm}

\noindent
The set $\mcal{S}$ is infinite since 
elliptic curves over $\mbb{Q}$, 
whose coefficients of their defining equations 
are $p$-adically close enough to 
that of $E_{\pi}$
and also
$\ell$-adically close enough to 
that of $E_{(\ell)}$,
are contained in $\mcal{S}$.
Now we take any elliptic curve $E$ over $\mbb{Q}$ 
whose $\mbb{Q}$-isomorphism class is in the set $\mcal{S}$.
By the condition (c), $E[p]$ is irreducible as a representation
of $G_{\mbb{Q}}$.
By the classical Serre's modularity conjecture  (proved by Khare and Wintenberger)
and the well-known fact that $p$-adic representations arising from Hecke eigencusp forms of level prime to $p$
are crystalline,  
we know that $(E[p]\otimes_{\mbb{F}_p} \overline{\mbb{F}}_p)|_{G_{\mbb{Q}_p}}$ 
is the reduction of a lattice in some crystalline $\overline{\mbb{Q}}_p$-representation. 
Furthermore, by the condition (a) and  Proposition 5 (2) of \cite{Se},
we know that $(E[p])|_{G_{\mbb{Q}_p}}$  is torsion crystalline with Hodge-Tate weights in 
$[0,p]$.
Therefore, so is $E_q[p]$ by (b).
\end{proof}

\begin{proof}[Proof of Proposition \ref{Prop3}]
Put $T=E_{\pi}[p]$ and $T'=\mbb{F}_p\oplus \mbb{F}_p(1)$. 
We know that $T$ and $T'$ are in $\mrm{Rep}^{p}_{\mrm{tor}}(G_{\mbb{Q}_p})$ by Lemma \ref{Prop3lem}.
They are not isomorphic as representations of $G_{\mbb{Q}_p}$
but isomorphic as representations of $G_1$.
This gives the desired result.
\end{proof}

Here we suggest the following question. 

\begin{question}
\label{Que}
What is the necessary and sufficient condition  for that the functor 
\[
\mrm{Rep}^r_{\mrm{tor}}(G_K)\overset{\mrm{res}}{\longrightarrow} \mrm{Rep}_{\mrm{tor}}(G_{\infty}),\quad T\mapsto T|_{G_{\infty}}
\]
is fully faithful?
Furthermore, does this condition depend only on $e$ and $r$?
\end{question}

\begin{remark}
\label{Rem}
(1) We do not know whether the full faithfulness of the functor in Question  \ref{Que}
depends on the choice of the system $(\pi_n)_{n\ge 0}$ or not (see Proposition \ref{Prop2}).
However, it is not difficult to see the following:
Take two systems $(\pi_n)_{n\ge 0}$ and $(\pi'_n)_{n\ge 0}$
of $p^n$-th roots of a fixed uniformizer $\pi$ of $K$
(thus we have $\pi_0=\pi'_0=\pi$).
Put $K_{\infty}=\bigcup_{n\ge 0} K(\pi_n)$ (resp.\ $K'_{\infty}=\bigcup_{n\ge 0} K(\pi'_n)$)
and $G_{\infty}=\mrm{Gal}(\overline{K}/K_{\infty})$ 
(resp.\ $G'_{\infty}=\mrm{Gal}(\overline{K}/K'_{\infty})$).
Then, 
the restriction functor 
$\mrm{Rep}^r_{\mrm{tor}}(G_K)\overset{\mrm{res}}{\longrightarrow} \mrm{Rep}_{\mrm{tor}}(G_{\infty})$
is fully faithful if and only if 
the restriction functor 
$\mrm{Rep}^r_{\mrm{tor}}(G_K)\overset{\mrm{res}}{\longrightarrow} \mrm{Rep}_{\mrm{tor}}(G'_{\infty})$
is.
In fact, we can check this from the fact that 
$G_{\infty}$ and $G'_{\infty}$ are conjugate with each other by some element of $G_K$.

\noindent
(2) A torsion $\mbb{Z}_p$-representation of $G_K$ is called {\it finite flat} if
it is isomorphic to $G(\overline{K})$ as $\mbb{Z}_p$-representations of $G_K$ for some $p$-power order 
finite flat commutative group scheme $G$ over the integer ring of $K$.
If $r=1$, then the category $\mrm{Rep}^r_{\mrm{tor}}(G_K)=\mrm{Rep}^1_{\mrm{tor}}(G_K)$ coincides with the category
of finite flat representations of $G_K$ 
(this can be checked by, for example, Theorem 3.1.1 of \cite{BBM}). 
Breuil proved in Theorem 3.4.3 of \cite{Br2} that the restriction functor 
$\mrm{Rep}^1_{\mrm{tor}}(G_K)\overset{\mrm{res}}{\longrightarrow} 
\mrm{Rep}_{\mrm{tor}}(G_{\infty})$ is fully faithful for any $K$ without any restriction on $e$.
In fact, this assertion is true even if $p=2$ 
(cf.\ \cite{Kim}, \cite{La}, \cite{Li3}, proved independently. 
 Explicitly, see Corollary 4.4 of \cite{Kim}). 

\noindent
(3) If $e=1$ and $r<p-1$, then the fact that the restriction functor 
$\mrm{Rep}^r_{\mrm{tor}}(G_K)\overset{\mrm{res}}{\longrightarrow} 
\mrm{Rep}_{\mrm{tor}}(G_{\infty})$ is fully faithful has been already known (\cite{Br1}, the proof of Th\'eor\`em 5.2).

\noindent
(4) Observing known results as above and results shown in this paper, 
it seems that the answer of Question \ref{Que}
should be ``$e(r-1)<p-1$''.
\end{remark}

\appendix

\section{$(\vphi, \hat{G})$-modules associated with crystalline representations}

In Proposition 5.9 of \cite{GLS},
a necessary condition for representations arising 
from free $(\vphi,\hat{G})$-modules to be crystalline is given.
In this appendix, we show that the converse holds.
The result here justifies the subscript ``cris'' of the category 
$\mrm{Mod}^{r,\hat{G},\mrm{cris}}_{/\mfS_{\infty}}$ defined in Section 3. 

We continue to use the same notation as in Section 2.
For any integer $n\ge 0$, we define ideals of $W(R)$ as below:
\[ 
I^{[n]}W(R):=\{a\in W(R) ; \vphi^m(a)\in \mrm{Fil}^nA_{\mrm{cris}}\
{\rm for\ every}\ m\ge 0 \},\ 
I^{[n^+]}W(R):=I^{[n]}W(R)I_+W(R)
\] 
(see Section 5 of \cite{Fo2} for more precise information).
The proof of Lemma 3.2.2 of \cite{Li2} shows that 
$I^{[n]}W(R)$ is a principal ideal of $W(R)$ generated by $\vphi(\mft)^n$.
In particular we see that 
$u^p\vphi(\mft)$ is contained in $I^{[1^+]}W(R)=\vphi(\mft)I_+W(R)$.
Recall that $\hat{T}(\hat{\mfM})\otimes_{\mbb{Z}_p} \mbb{Q}_p$ is a 
semi-stable $\mbb{Q}_p$-representation of $G_K$ (Theorem \ref{Th1} (2))
and $\tau(x)-x\in I_+W(R)\otimes_{\vphi,\mfS} \mfM$
for any $x\in \mfM$.  
The main purpose of this appendix is to prove the following:

\begin{theorem}
\label{app}
Let $\hat{\mfM}\in \mrm{Mod}^{r,\hat{G}}_{/\mfS}$ be a $(\vphi,\hat{G})$-module.
The followings are equivalent:
\vspace{-2mm}
\begin{enumerate}
\item[$(1)$] $\hat{T}(\hat{\mfM})\otimes_{\mbb{Z}_p} \mbb{Q}_p$ is crystalline.  

\vspace{-2mm}
           
\item[$(2)$] For any $x\in \mfM$, we have
$
\tau(x)-x\in I^{[1^+]}W(R)\otimes_{\vphi,\mfS} \mfM.
$

\vspace{-2mm}           
           
\item[$(3)$] For any $x\in \mfM$, we have
$
\tau(x)-x\in u^p\vphi(\mft)(W(R)\otimes_{\vphi,\mfS} \mfM).
$
\end{enumerate}
\end{theorem}

Before giving a proof of this theorem,
we shall recall some known facts about $(\vphi,\hat{G})$-modules.
Let $\hat{\mfM}\in \mrm{Mod}^{r,\hat{G}}_{/\mfS}$ be a $(\vphi,\hat{G})$-module, 
and put 
$\mcal{D}=S_{K_0}\otimes_{\vphi, \mfS}\mfM$.
Then $\mcal{D}$ has a structure as a Breuil module which corresponds 
to the semi-stable representation $\hat{T}(\hat{\mfM})\otimes_{\mbb{Z}_p} \mbb{Q}_p$ of $G_K$.
(Breuil modules here are objects of ``$\mcal{MF}_S(\vphi, N)$'' defined in Section 6.1 of \cite{Br0}.
 It is useful for the reader to refer also Section 5 of \cite{Li1}.)
Denote by $N_{\mcal{D}}$ the monodromy operator of $\mcal{D}$ and 
define a $G_K$-action on 
$B^{+}_{\mrm{cris}}\otimes_S \mcal{D}
=B^{+}_{\mrm{cris}}\otimes_{\vphi,\mfS} \mfM$ by 
\[
g(a\otimes x)=\sum^{\infty}_{i=0}g(a)\gamma_i(-\mrm{log}([\underline{\e}(g)]))\otimes N^i_{\mcal{D}}(x)
\]
for $g\in G_K, a\in B^+_{\mrm{cris}}, x\in \mcal{D}$.
Here, $\underline{\e}(g):=g(\underline{\pi})/\underline{\pi}\in R$.
By the construction of the quasi-inverse of the functor $\hat{T}$ of Theorem \ref{Th1} (2) (\cite{Li2}, Section 3.2), 
this $G_K$-action is stable on $\whR\otimes_{\vphi,\mfS} \mfM\subset B^{+}_{\mrm{cris}}\otimes_{\vphi,\mfS} \mfM$
and it factors through $\hat{G}$, which gives the original $\hat{G}$-action of 
the $(\vphi,\hat{G})$-module $\hat{\mfM}$. 
From now on we put $t=-\mrm{log}([\underline{\e}(\tau)])$.
For any $n\ge 0 $ and any $x\in \mcal{D}$,
an induction on $n$ shows that 
\[
(\tau-1)^n(x)=\sum^{\infty}_{m=n}(\sum_{i_1+\cdots i_n=m, i_j\ge 0}\frac{m!}{i_1!\cdots i_n!})
\gamma_m(t)\otimes N^m_{\mcal{D}}(x)
\in B^{+}_{\mrm{cris}}\otimes_S \mcal{D}
\] 
and in particular $\frac{(\tau-1)^n}{n}(x)\to 0$ $p$-adically as $n\to \infty$.
Hence we can define 
\[
\mrm{log}(\tau)(x)=
\sum^{\infty}_{n=1}(-1)^{n-1}\frac{(\tau-1)^n}{n}(x) \in B^{+}_{\mrm{cris}}\otimes_S \mcal{D}.
\]
It is not difficult to check the equation $\mrm{log}(\tau)(x)=t\otimes N_{\mcal{D}}(x)$.
Consequently the monodromy operator $N_{\mcal{D}}$ can be reconstructed from the $\tau$-action of $\hat{\mfM}$ by the relation 
$\frac{1}{t}\mrm{log}(\tau)(x)=N_{\mcal{D}}(x)$.
Put $D=\mcal{D}/I_+S_{K_0}\mcal{D}$.
Then $D$ has a structure as a filtered $(\vphi,N)$-module over $K_0$ which corresponds to 
$\hat{T}(\hat{\mfM})\otimes_{\mbb{Z}_p} \mbb{Q}_p$
and the monodromy operator $N_D$ of $D$ is given by $N_{\mcal{D}}$ mod $I_+S_{K_0}\mcal{D}$
(\cite{Br0}, Section 6).
Hence  $\hat{T}(\hat{\mfM})\otimes_{\mbb{Z}_p} \mbb{Q}_p$ is crystalline 
if and only if $N_{\mcal{D}}$ mod $I_+S_{K_0}\mcal{D}$ is zero.

\begin{proof}[Proof of Theorem \ref{app}]
The implication $(1)\Rightarrow (3)$ follows from Proposition 5.9 of \cite{GLS}.
It is clear that $(3)$ implies $(2)$.
Thus it suffices to show the implication $(2)\Rightarrow (1)$.
Assume the condition $(2)$.
We use the same notation $\mcal{D}, N_{\mcal{D}}, D, N_D$ as the above.
We often regard $\mfM$ as a $\vphi(\mfS)$-submodule of $\mcal{D}$.
 Let $x\in \mfM$.
For any integer $n> 0$, it is shown in the proof of Proposition 2.4.1 of \cite{Li2.5} that
\begin{itemize}
\item[(A)] $(\tau-1)^n(x)\in I^{[n]}W(R)\otimes_{\vphi,\mfS} \mfM$;
\item[(B)] $\frac{(\tau-1)^n}{nt}(x)$ is well-defined in $A_{\mrm{cris}}\otimes_{\vphi,\mfS} \mfM$
and $\frac{(\tau-1)^n}{nt}(x)\to 0$ $p$-adically as $n\to \infty$. Therefore,  
we have $\frac{1}{t}\mrm{log}(\tau)(x)\in A_{\mrm{cris}}\otimes_{\vphi,\mfS} \mfM\subset B^+_{\mrm{cris}}\otimes_{\vphi, \mfS} \mfM$.
\end{itemize}
By (A), we can take $y_n\in W(R)\otimes_{\vphi,\mfS}\mfM$ such that 
$(\tau-1)^n(x)=\vphi(\mft)^ny_n$.
Then we have the equation
\[
(\ast):\ cN_{\mcal{D}}(x)=c\cdot\frac{1}{t}\mrm{log}(\tau)(x)=\frac{\tau-1}{\vphi(\mft)}(x)
+\sum^{\infty}_{n=2}(-1)^{n-1}\frac{\vphi(\mft)^{n-1}}{n}y_n.
\]
Here $c=\frac{t}{\vphi(\mft)}$, which is a unit of  $A_{\mrm{cris}}$ (\cite{Li2}, Example 3.2.3).
Note that $\frac{\tau-1}{\vphi(\mft)}(x)$ is contained in $I_+W(R)\otimes_{\vphi,\mfS} \mfM$ by the assumption $(2)$.

Now we claim that 
there exists an integer $n_0>1$ such that $\frac{(n-2)!}{n}$ is in $\mbb{Z}_p$ for any  $n>n_0$.
Admitting this claim, we proceed a proof of Theorem \ref{app}.
Consider the decomposition 
\[
\sum^{\infty}_{n=2}(-1)^{n-1}\frac{\vphi(\mft)^{n-1}}{n}y_n=
\vphi(\mft)\sum^{n_0}_{n=2}(-1)^{n-1}\frac{\vphi(\mft)^{n-2}}{n}y_n+
\vphi(\mft)\sum^{\infty}_{n=n_0+1}(-1)^{n-1}\frac{\vphi(\mft)^{n-2}}{n}y_n.
\] 
By the claim, we see that 
$\frac{\vphi(\mft)^{n-2}}{n}=
\frac{\vphi(\mft)^{n-2}}{(n-2)!}\cdot \frac{(n-2)!}{n}
=c^{-(n-2)}\gamma_{n-2}(t)\frac{(n-2)!}{n}
$
is contained in $A_{\mrm{cris}}$ for any $n>n_0$ and it goes to zero $p$-adically as $n\to \infty$. 
In particular, (the first term and) the second term of the above decomposition are
contained in $\vphi(\mft)(B^+_{\mrm{cris}}\otimes_{\vphi,\mfS} \mfM)$,
which is contained in $I_+B^+_{\mrm{cris}}\otimes_{\vphi,\mfS} \mfM$.
Hence $\sum^{\infty}_{n=2}(-1)^{n-1}\frac{\vphi(\mft)^{n-1}}{n}y_n$ is also
contained in $I_+B^+_{\mrm{cris}}\otimes_{\vphi,\mfS} \mfM$.
Note that $\nu(c)=1$ since $c=\frac{t}{\vphi(\mft)}=\prod^{\infty}_{n=0} \vphi^n(\frac{c^{-1}_0E(u)}{p})$
and $\nu(u)=0$,
and furthermore $\nu(\mft)=0$.
Therefore, by $(\ast)$ modulo $I_+B^+_{\mrm{cris}}\otimes_{\vphi, \mfS} \mfM$,
we obtain the relation 
$N_D(\bar{x})=0$ in 
$D=\mcal{D}/I_+S_{K_0}\mcal{D}\subset 
(B^+_{\mrm{cris}}\otimes_{\vphi, \mfS} \mfM)/(I_+B^+_{\mrm{cris}}\otimes_{\vphi, \mfS} \mfM)$
where $\bar{x}$ is the residue class of $x$.
Since the image of $\mfM$ in $D=\mcal{D}/I_+S_{K_0}\mcal{D}$ generates 
$D$ as a $K_0$-vector space, 
we obtain that $N_D=0$. This implies $(1)$.
Hence it suffices to show the claim. Let $v_p$ be the $p$-adic valuation with $v_p(p)=1$.
For any positive integer $n$, write $n=p^sm$ with $p\not{|}m$.
If $s=0$, it is clear that $\frac{(n-2)!}{n}\in \mbb{Z}_p$.
Suppose $s\ge 1$.
If $m\ge 2$, we have 
$v_p((n-2)!)\ge v_p((2p^s-2)!)\ge v_p(p^s!)\ge s=v_p(n)$.
If $m=1$ and $s\ge 3$,
we have $v_p((n-2)!)\ge v_p(p^{s-1}!)=\frac{1}{2}s(s-1)\ge s=v_p(n)$.
This finishes the proof. 
\end{proof}

\end{document}